\documentclass[11pt,letterpaper]{amsart}

\usepackage{amsfonts}
\usepackage{amssymb}
\usepackage{amsmath}
\usepackage{amsthm}
\usepackage{enumerate}
\usepackage[all]{xy}
\usepackage{graphicx}
\usepackage{appendix}
\usepackage{hyperref}

\newtheorem{theorem}{Theorem}
\newtheorem{corollary}{Corollary}

\newtheorem{remark}{Remark}
\newtheorem{conjecture}{Conjecture}
\newtheorem{example}{Example}
\newtheorem{proposition}{Proposition}

\newtheorem{thma}{Theorem}

\theoremstyle{definition}
\newtheorem{definition}{Definition}

\def\C{\mathbb C}
\def\R{\mathbb R}
\def\T{\mathbb T}
\def\N{\mathbb N}
\def\ba{\mathbf a}
\def\bc{\mathbf c}
\def\Z{\mathbb Z}
\def\Q{\mathbb Q}
\def\SF{\mathcal{SF}}

\makeatletter
\@namedef{subjclassname@2020}{%
  \textup{2020} Mathematics Subject Classification}
\makeatother

\begin{document}

\title[Time-Average]{Time-Averages along Partial Orbits for Dynamical Systems}

\author{Yunping Jiang and Jessica Liu}


\address{Yunping Jiang: Department of Mathematics\\
Queens College of the City University of New York\\
Flushing, NY 11367-1597 \&
Department of Mathematics\\
Graduate School of the City University of New York\\
365 Fifth Avenue, New York, NY 10016}
\email{yunping.jiang@qc.cuny.edu}

\address{Jessica Liu: 
Department of Mathematics\\
Graduate School of the City University of New York\\
365 Fifth Avenue, New York, NY 10016}
\email[]{yliu11@gradcenter.cuny.edu}
\subjclass[2020]{Primary 11K65, 37A44; Secondary 11N37, 37A30}

\keywords{uniformly behaved in sequence in $\N$, $\ba$-minimal mean Lyapunov stable, $\ba$-minimal mean attractable, time-average along a sequence, convergence for every point}


\begin{abstract}
We investigate the convergence of time-averages of continuous functions along partial orbits for continuous dynamical systems on compact metric spaces. Unlike the classical focus of ergodic theory on convergence almost everywhere, our study emphasizes convergence at every point, motivated by problems in number theory. To this end, we consider a sequence having a uniform behavior property in the natural numbers. We also define and study two structural properties of dynamical systems along such partial orbits: minimal mean Lyapunov stability and minimal mean attractability. Our main result shows that if a continuous dynamical system is both minimally mean Lyapunov stable and minimally mean attractable along a given uniformly behaved sequence, and is minimally uniquely ergodic, then the time-average of any continuous function along the corresponding partial orbit converges to the space average at every point.  
\end{abstract}

\maketitle

\section{\bf Introduction}

Suppose \( f: X \to X \) is a measure-preserving transformation on a probability measure space \( (X, \mathcal{B}, \nu) \); that is, \( f \) is \( \nu \)-measurable and
\[
\nu(f^{-1}(A)) = \nu(A), \quad \forall\, A \in \mathcal{B}.
\]
We also call \( \nu \) an \( f \)-invariant probability measure.

Let \( \N \) denote the set of all natural numbers. Consider the dynamical system \( \{f^{n}\}_{n \in \{0\} \cup \N} \), where \( f^{n} \) denotes the \( n \)-fold composition of \( f \) for \( n \in \N \), and \( f^{0} \) denotes the identity. We use \( f \) to denote this dynamical system. We say \( f \) is \emph{ergodic} if for any \( A \in \mathcal{B} \) satisfying \( f^{-1}(A) = A \), we have \( \nu(A) = 0 \) or \( 1 \).

Given a \( \nu \)-integrable function \( \phi \), we define the time-average and space-average as
\[
S_{N} \phi : = \frac{1}{N} \sum_{n=1}^{N} \phi \circ f^{n} \quad \text{and} \quad \overline{\phi}: = \int_{X} \phi \, d\nu.
\]
One of the most celebrated theorems in ergodic theory is the Birkhoff ergodic theorem:

\medskip
\begin{thma}~\label{bt}
If \( f \) is ergodic, then
\[
\lim_{N \to \infty} S_{N} \phi (x) = \overline{\phi}, \quad \nu\text{-a.e. } x \in X.
\]
\end{thma}

Motivated by problems in number theory and ergodic theory (see~\cite{AKPLR,BR,FH,FK,R,Sa1,Sa2,T}), rather than almost everywhere convergence, we are more interested in convergence for \emph{every point} (see~\cite{FJ,JNon,JPro,JPre}). Ergodic theory does provide such a result. We say $f: X\to X$ is uniquely ergodic if it has only one $f$-invariant probability measure, which is automatically ergodic.

\medskip
\begin{thma}~\label{uet}
Suppose \( X \) is a compact metric space and \( f: X \to X \) is a continuous and uniquely ergodic dynamical system. Then for any continuous function \( \phi \) on \( X \),
\[
\lim_{N \to \infty} S_{N} \phi(x) = \overline{\phi}, \quad \forall\, x \in X.
\]
\end{thma}

\medskip
\noindent\textbf{Standing Assumption:} \emph{From now on, we assume that \( X \) is a compact metric space, \( f: X \to X \) is a continuous dynamical system, and \( \phi \) is a continuous function defined on \( X \).}

\medskip
We will study time-averages along partial orbits, as defined below. This investigation is again motivated by problems in number theory (see~\cite{BR,D,P,R,Se}).

Let \( \mathbf{a} = (a_n)_{n \in \N} \) be a sequence taking values in \( \N \). Define the time-averages along \( \mathbf{a} \) as
\[
S_{N}^{\mathbf{a}} \phi = \frac{1}{N} \sum_{n=1}^{N} \phi \circ f^{a_n}.
\]
When \( a_n = P(n) \) for a polynomial \( P \in \mathbb{Z}[n] \), J.~Bourgain~\cite{B} studied the almost everywhere convergence of \( S_{N}^{\mathbf{a}} \phi \). However, such sequences are typically unsuitable for studying pointwise convergence for \emph{every} point (see \S4). For the purpose of studying the convergence for every point, we use a notation of \emph{uniformly behaved in $\N$} for $\ba$ (see Definition~\ref{ubs}).

We define the notions 
of \emph{minimal mean Lyapunov stability} (see Definitions~\ref{amls} and~\ref{ammls}) and 
\emph{minimal mean attractability} along \( \mathbf{a} \) 
(see Definition~\ref{amma}). We introduce the term of the minimal unique ergodicity (see Definition~\ref{mue}). 

The main result (Theorem~\ref{main}), which we will prove in this paper, says that for a uniformly behaved in $\N$ sequence \( \mathbf{a} \) and 
a continuous dynamical system \( f: X \to X \) on a compact metric space \( X \), if \( f \) is minimal mean Lyapunov stable along \( \mathbf{a} \), minimal mean attractable along \( \mathbf{a} \), and minimal uniquely ergodic, then for every continuous function \( \phi \) on \( X \),
\[
\lim_{N \to \infty} S_N^{\mathbf{a}} \phi(x) = \int_K \phi \; d\nu_K, \quad \forall\; x\in X,
\]
where \( K \subset X \) is the minimal subset of $X$ such that $x$ is mean attractable to $K$ and \( \nu_K \) is the unique \( (f|K) \)-invariant probability measure.

The paper is organized as follows. In Section~\ref{mainsec}, we begin by introducing the notions of $\ba$-minimal Lyapunov stable dynamical systems ($\ba$-MMLS; Definition~\ref{ammls}) and $\ba$-minimal mean attractable dynamical systems ($\ba$-MMA; Definition~\ref{amma}). We then define what it means for a sequence taking values in $\N$ to be uniformly behaved in $\N$ (Definition~\ref{ubs}). With these definitions in place, we state our main result (Theorem~\ref{main}) and proceed to prove it in the same section through a three-step argument.

In Section~\ref{ammlsmma}, we study $\ba$-MMLS and $\ba$-MMA dynamical systems. We begin by noting that equicontinuous systems provide trivial but still important examples. We then construct two non-trivial examples when $\ba$ is a subsequence of the natural numbers with positive lower density: one is a Denjoy counterexample in the setting of circle homeomorphisms (Theorem~\ref{den}), and the other is a continuous interval map with zero topological entropy (Theorem~\ref{int}). We also raise several open questions in Remark~\ref{qu} and propose a conjecture (Conjecture~\ref{oden}).

The sequence of all natural numbers is a basic example of a uniformly behaved in $\N$ sequence. A classical example is given by the prime Omega function, which also generates a uniformly behaved in $\N$ sequence. In Section~\ref{eubs}, we establish a connection between oscillating sequences taking values in $\{-1,1\}$ and uniformly behaved in $\N$ sequences (Theorems~\ref{ldt} and~\ref{ubos}). This connection allows us to construct further examples, including the sequence indexed by the Thue–Morse sequence (Example~\ref{tms}), the sequence indexed by the Rudin–Shapiro sequence (Example~\ref{rss}), and the sequences of natural numbers with an even (or odd) number of prime factors (Examples~\ref{ef} and~\ref{of}).

In Section~\ref{sfn}, we investigate the square-free uniformly behaved property for sequences taking values in $\N$. We show that any completely additive uniformly behaved sequence in $\N$ is also square-free uniformly behaved in $\N$ (Theorem~\ref{sfubt}). Combined with the main theorem, this result leads to another principal conclusion of the paper (Theorem~\ref{fcor2}). Finally, we demonstrate that although the sequence of square-free numbers is uniformly distributed modulo $1$, it is not uniformly behaved in $\N$.

\medskip
\medskip
\noindent
\textbf{Acknowledgment:} This research is partially supported by a grant from the Simons Foundation (\#942077), an enhanced research award from the PSC-CUNY research program (\#66680-00 54).

\medskip
\medskip

\section{\bf The Statement of the Main Theorem And the Proof}~\label{mainsec}

We first introduce some new concepts that we will use to state our main theorem. 

Suppose $\ba=(a_{n})_{n\in \N}$ is a sequence taking values in $\N$. Let $E$ be a subset of $\mathbb{N}$. 

\medskip
\begin{definition}
The $\ba$-upper density and the $\ba$-lower density of $E$ is defined as
$$
\overline{D}_{\ba}(E) = \limsup_{N\to\infty} \frac{\# \{ n\mid a_{n}\in  E, \; 1\leq n\leq N\}}{N}
$$
and 
$$
\underline{D}_{\ba}(E) = \liminf_{N\to\infty} \frac{\# \{ n\mid a_{n}\in E, 1 \leq n\leq N\}}{N}.
$$
When $\overline{D}_{\ba}(E)=\underline{D}_{\ba}(E)$, we call this number $\ba$-density of $E$ and denote as $D_{\ba} (E)$.
\end{definition}

\medskip\begin{remark}
When $\ba =(n)_{n\in \N}$ is the sequence of all natural numbers, $\overline{D}_{\ba}(E)$, $\underline{D}_{\ba}(E)$, and $D_{\ba}(E)$ are just usual the upper density $\overline{D}(E)$, the lower density $\underline{D}(E)$, and the density $D(E)$ of $E$ in $\N$.
\end{remark}

Suppose $f:X\to X$ is a continuous dynamical system on a compact metric space $X$. 
Then there is at least one $f$-invariant probability measure. All $f$-invariant probability measures form a non-empty, convex, and weakly compact metrizable set. All ergodic $f$-invariant probability measures are extremal points of this convex set. If there is only one $f$-invariant probability measure, then this measure must be ergodic. In this case, we call $f$ uniquely ergodic. 

A set \( K \subseteq X \) 
is called $f$-minimal if for every \( x \in K \),
\[
K = \overline{\{f^n(x)\}_{n \in \{0\} \cup \N}}.
\]

\medskip
\begin{definition}~\label{mue}
We say \( f \) is \emph{minimal uniquely ergodic} 
if \( f|_K: K \to K \) is uniquely ergodic 
for every $f$-minimal set \( K \).
\end{definition}

\medskip
\begin{definition}~\label{amls}
We say $f$ is mean Lyapunov stable along $\ba$ (abbreviated $\ba$-MLS) if for every $\epsilon > 0$, there is a $\delta > 0$ such that for any $x,y\in X$ with $d(x,y)<\delta$, there is an subset $E=E(\epsilon,x,y)\subset \N$ with $\overline{D}_\ba(E)<\epsilon$ such that 
$$
d(f^{a_{n}} x,f^{a_{n}}y) <\epsilon, \quad \forall\; n\in \mathbb{N} \hbox{ such that } a_{n}\not\in E.
$$ 
\end{definition}

\medskip
\begin{remark}
When $\ba =(n)_{n\in \N}$ is the sequence of all natural numbers, $\ba$-MLS is the mean Lyapunov stable (abbreviated MLS) defined in~\cite{Fo}.
\end{remark}

\medskip
\begin{definition}~\label{ammls}
We say $f$ is minimal mean Lyapunov stable along $\ba$ (abbreviated $\ba$-MMLS) if $f: K\to K$ is $\ba$-MLS for every $f$-minimal set. 
\end{definition}

\medskip
\begin{definition}~\label{amma}
Suppose $K$ is an $f$-minimal set.  A point $x\in X$ is said to be meanly attracted to $K$ along $\ba$ if for any $\epsilon >0$, there is a point $z=z_{\epsilon, x}\in K$ 
such that
$$
\limsup_{N\to\infty} \frac{1}{N} \sum_{n=1}^{N} d(f^{a_{n}} x, f^{a_{n}} z) <\epsilon.
$$
The {\em basin of mean attraction} of $K$ along $\ba$, denoted as $\hbox{\rm Basin}_{\ba}(K)$, is the set of all
points $x$ which are meanly attracted to $K$ along $\ba$. 
We call $f$ minimal mean attractable (abbreviated $\ba$-MMA) if
$$
   X= \bigcup_{K} \mbox{\rm Basin}_{\ba}(K)
$$
where $K$ varies among all $f$-minimal sets in $X$.
\end{definition}
 
 \medskip
 \begin{definition}~\label{ubs}
 We say $\ba$ is uniformly behaved in $\N$ if for any $0<\theta <1$,
$$
 \lim_{N\to \infty} \frac{1}{N} \sum_{n=1}^{N} e^{2\pi i a_{n} \theta} =0.
$$
 \end{definition}

\medskip
\begin{remark}
Definition~\ref{ubs} closely relates to the Hartman uniform distribution defined on a locally compact abelian group $G$ (refer to~\cite[Chapter 4, Definition 5.6]{KN}). Applying this abstract definition to the additive integral group $\Z$ (refer to~\cite[Chapter 4, Example 5.11]{KN}), the Hartman uniform distribution on $\Z$ unifies two uniform distribution notations into one. The first is the uniform distribution modulo $1$ and the second is the uniform distribution in $\mathbb{Z}$ (refer to Section~\ref{eubs}). When considering $\ba$ as a subset in the additive group $\Z$, a uniformly behaved in $\N$ sequence $\ba$ can also be called a Hartman uniformly distributed in $\Z$ sequence.  
\end{remark}

Now, let us state and prove our main theorem. 

Let $C(X)$ denote the space of all continuous functions $\phi$ on $X$ with the maximal norm, 
$$
\|\phi\| =\max_{x\in X} |\phi (x) |.
$$
Let $d(\cdot,\cdot)$ denote the metric on $X$. 

\medskip
\begin{theorem}[Main Theorem]\label{main}
Suppose \( \mathbf{a} \) is uniformly behaved in $\N$ and suppose \( f: X \to X \) is $\ba$-MMLS, $\ba$-MMA, and minimal uniquely ergodic. Then, for any $\phi \in C(X)$, 
$$
\lim_{N \to \infty} S_N^{\mathbf{a}} \phi(x) = \int_K \phi \; d\nu_K, \quad \forall\; x\in X,
$$
where $K\subset X$ is the $f$-minimal set such that $x\in \mbox{\rm Basin}_{\ba}(K)$ and
\( \nu_K \) is the unique \( (f|K) \)-invariant probability measure.
\end{theorem}

\begin{proof} 
The first step in the proof is to show that for a given $f$-minimal set $K$, if for every point $z\in K$, the sequence 
$\{S_{N}^{\ba} (z)\}_{N=1}^{\infty}$ converges and has the same limit $L$, then for every point $x\in \mbox{\rm Basin}_{\ba}(K)$, the sequence 
$\{S_{N}^{\ba} (x)\}_{N=1}^{\infty}$ converges to $L$ as well. We prove this claim. 

Given a point $x\in \mbox{\rm Basin}_{\ba}(K)$. For every $\epsilon >0$, we have a point $z\in K$ such that
$$
\limsup_{N\to\infty} \frac{1}{N} \sum_{n=1}^{N} d(f^{a_{n}} x, f^{a_{n}} z) <\epsilon^{2}.
$$
Let 
$$
E=\{ n\mid d(f^{a_{n}}(x), f^{a_{n}}(z)) \geq \epsilon\}.
$$
Then we have an integer $N_{0}\geq 1$ such that for $N>N_{0}$,
$$
\frac{\epsilon \cdot \#\{n \mid a_{n}\in E, \; 1\leq n\leq N \}}{N}
$$
$$
\leq 
\frac{1}{N} \sum_{1\leq n\leq N, a_{n}\in E} d(f^{a_{n}} x, f^{a_{n}} z) + \frac{1}{N} \sum_{1\leq n\leq N, a_{n}\not\in E} d(f^{a_{n}} x, f^{a_{n}} z)< \epsilon^{2},
$$
which implies that 
$$
\frac{\#\{n \mid a_{n}\in E, \; 1\leq n\leq N \}}{N}<\epsilon, \quad \forall\; N>N_{0}. 
$$
Since $\phi$ is uniformly continuous on $X$, we have a $\delta>0$ such that for any $w, y\in X$ with $d(w,y)<\delta$,
$$
|\phi (w)-\phi(y) <\epsilon.
$$
This implies that for any $N>N_{0}$, 
\begin{align*}
\Big| S_{N}^{\ba} \phi (x)-S_{N}^{\ba}\phi (z) \Big| 
&\leq 
\frac{1}{N} \sum_{1\leq n\leq N, a_{n}\not\in E} |\phi (f^{a_n} (x))-\phi (f^{a_{n}}(z))| \\
&+\frac{1}{N} \sum_{1\leq n\leq N, a_{n}\in E} |\phi (f^{a_n} (x))-\phi (f^{a_{n}}(z))|\\
&\leq \epsilon +2 \|\phi\| \frac{\#\{ n\mid a_{n}\in E, \; 1\leq n\leq N\}}{N}<\epsilon +2\epsilon \|\phi\|.
\end{align*}
Since 
$$
L=\lim_{N\to \infty} S_{N}^{\ba} \phi (z),
$$
we have an integer $N_{1}\geq N_{0}$ such that 
$$
\Big| S_{N}^{\ba}\phi (z) -L \Big| <\epsilon, 
\quad \forall\; N> N_{1}.
$$
Then we have that for $N>N_{1}$, 
\begin{align*}
\Big| S_{N}^{\ba}\phi (x) -L \Big| 
&\leq \Big| S_{N}^{\ba}\phi (x)- S_{N}^{\ba}\phi (z)\Big| +
\Big| S_{N}^{\ba}\phi (z) -L\Big|\\
& \leq 2 \epsilon (1 + \|\phi\|),
\end{align*}
which implies that 
$$
\lim_{N\to \infty} S_{N}^{\ba} \phi (x) =L,
$$
which completes the proof of the first step. 

The second step in the proof is to show that the family 
$\mathcal{F}=\{S_{N}^{\ba}\phi\}_{N=1}^{\infty}$ 
is precompact on any $f$-minimal set $K$, that is, 
any sequence in $\mathcal{F}$ has a convergent subsequence in $C(K)$. We prove this claim. 

It is clear if $\phi\equiv 0$. Assume $\|\phi\|>0$. 
Since $\phi$ is uniformly continuous on $K$, for any $\epsilon >0$ there exists $\delta >0$ such that for any $x, y\in K$ with $d(x, y)< \delta$, 
$$
 |\phi(x)- \phi(y)|<\frac{\epsilon}{2}.
$$
Take $\delta <\frac{\epsilon}{4\|\phi\|}$.
Since $f$ is $\ba$-MLS on $K$, there is a $\eta > 0$ such that $x, y\in X$ with $d(x, y) <\eta$, we have a set $E\subset \N$ with $D_{\ba}(E) <\delta$, such that 
$$
d(f^{a_{n}} x,f^{a_{n}} y) <\delta, \quad \forall\; n\in \N \hbox{ such that } a_{n}\not\in E.
$$
Thus we have an integer $N_{0}\geq 1$ such that for $N\ge N_{0}$ we have
\begin{align*}
|S_{N}^{\ba}\phi(x) - S_{N}^{\ba} \phi(y)|
&\leq \frac{1}{N} \sum_{1\leq n\leq N, a_{n}\not\in E} |\phi (f^{a_n} (x))-\phi (f^{a_{n}}(y))| \\
&+\frac{1}{N} \sum_{1\leq n\leq N, a_{n}\in E} |\phi (f^{a_n} (x))-\phi (f^{a_{n}}(y))|\\
&< \frac{\epsilon}{2} + 2\|\phi\| \frac{\#\{ n\mid a_{n}\in E, \; 1\leq n\leq N\}}{N} \leq \frac{\epsilon}{2} +2\|\phi\| \delta\\
&\leq \frac{\epsilon}{2} +\frac{\epsilon}{2}=\epsilon.
\end{align*}
This implies that $\mathcal{F}$ is an equicontinuous family. Since 
$$
\| S_{N}^{\ba} \phi \| \leq \|\phi\|, \quad \forall\; N\geq 1, 
$$
$\mathcal{F}$ is a uniformly bounded family. Therefore, the Ascoli-Arzela theorem implies that $\mathcal{F}$ is precompact in $C(K)$. 
This completes the proof of the second step.

The final step is to show that $\{ S_{N}^{\ba} \phi\}_{N=1}^{\infty}$ on any $f$-minimal set $K$ converges to the same limit 
$$
L=\int_{K} \phi \;d\nu_{K}.
$$
We prove this claim.

We use the notation,
$$
<\phi, \psi> =\int_{K} \phi \psi \; d\nu_{K}, \quad \phi, \psi \in C(K),
$$
where $\nu_{K}$ is the unique $(f|K)$-invariant probability measure. 

The statement that the measure $\nu_{K}$ is $(f|K)$-invariant is equivalent to say that
$$
<\phi\circ f, 1> =<\phi, 1>,\quad \forall\; \phi \in C(K).
$$

Let ${\mathbb T}=\{ z\in {\mathbb C} \;|\; |z|=1\}$ be the unit circle in the complex plane ${\mathbb C}$. 
Without loss of generality, assume that $\phi$ is real-valued. 
Consider the sequence of functions $\{ \phi\circ f^{n}\}_{n=0}^{\infty}$ in $C(K)$. From the spectral lemma (see~\cite[pp.94-95]{K}) in ergodic theory, we have a probability measure $m_{\phi}$ on ${\mathbb T}$ such that 
$$
<\phi\circ f^{n}, \phi> =\int_{K} \phi\circ f^{n} \phi \; d\nu_{K} = \int_{\mathbb T} e^{2\pi i n\theta} \; dm_{\phi}, \quad \forall\; n\geq 1, 
$$
which gives that 
$$
<(S_{N}^{\ba}\phi)\circ f-S_{N}^{\ba}\phi, \phi> =\int_{\mathbb T} (e^{2\pi i\theta} -1)\Big( \frac{1}{N}\sum_{n=1}^{N} e^{2\pi i a_{n}\theta}\Big) \; dm_{\phi}.
$$  

Consider 
\begin{align*}
& \hspace{15pt} \|(S_{N}^{\ba}\phi)\circ f-S_{N}^{\ba}\phi\|_{L^{2}(\nu_{K})}^{2}\\
&=<(S_{N}^{\ba}\phi)\circ f-S_{N}^{\ba}\phi, (S_{N}^{\ba}\phi)\circ f-S_{N}^{\ba}\phi>\\
&= \int_{X} \Big( (S_{N}^{\ba}\phi)\circ f-S_{N}^{\ba}\phi\Big) \Big( (S_{N}^{\ba}\phi)\circ f-S_{N}^{\ba}\phi\Big) \; d\nu_{K}.
\end{align*}
For each term $\phi\circ f^{b_{m}}$ in $(S_{N,\ba}\phi)\circ f-S_{N,\ba}\phi$, where $b_{m}=a_{m}+1$ or $a_{m}$,
we have 
\begin{align*}
&\hspace{15pt} \int_{X} \Big( (S_{N,\ba}\phi)\circ f-S_{N,\ba}\phi\Big) \phi\circ f^{b_{m}} \; d\nu_{K} \\
&=  \int_{X} \Big( (S_{N,\ba}\phi)\circ f\circ f^{-b_{m}} -S_{N,\ba}\phi\circ f^{-b_{m}} \Big)\phi\;d\nu_{K}\\
&=\int_{\T} \frac{1}{N}  \sum_{n=1}^{N} \Big( e^{2\pi i (a_{n}+1-b_{m}) \theta}-e^{2\pi i (a_{n}-b_{m})\theta}\Big) dm_{\phi}\\
&=\int_{\T} (e^{2\pi i \theta} -1)  \Big( \frac{1}{N}  \sum_{n=1}^{N} e^{2\pi i a_{n} \theta}\Big) \overline{e^{2\pi i b_{m}\theta}} dm_{\phi}.
\end{align*}
Summing up every term, we get 
\begin{align*}
&\hspace{15pt} \int_{X} \Big( (S_{N,\ba}\phi)\circ f-S_{N,\ba}\phi\Big) \Big( (S_{N,\ba}\phi)\circ f-S_{N,\ba}\phi\Big) \; d\nu_{K}\\
&=\int_{\T} (e^{2\pi i \theta} -1)  \Big( \frac{1}{N}  \sum_{n=1}^{N} e^{2\pi i a_{n} \theta}\Big) \Big(\overline{(e^{2\pi i \theta} -1)  \frac{1}{N}  \sum_{n=1}^{N} e^{2\pi i a_{n} \theta}}\Big) dm_{\phi}\\
&=\|(e^{2\pi i\theta} -1)\Big( \frac{1}{N}\sum_{n=1}^{N} e^{2\pi i a_{n}\theta}\Big)\|_{L^{2}(m_{\phi})}^{2}.
\end{align*}
Thus, we have that for $N\geq 1$,
$$
\|(S_{N,\ba}\phi)\circ f-S_{N,\ba}\phi\|_{L^{2}(\nu_{K})} =\|(e^{2\pi i\theta} -1)\Big( \frac{1}{N}\sum_{n=1}^{N} e^{2\pi i a_{n}\theta}\Big)\|_{L^{2}(m_{\phi})}.
$$
Since 
$$
\sup_{x\in X} |S_{N,\ba}\phi(fx)-S_{N,\ba}\phi(x)| \leq 2\|\phi (x)\|, \quad \forall \; N\geq 1
$$
and 
$$
\sup_{0\leq \theta<1} |(e^{2\pi i\theta} -1)\Big( \frac{1}{N}\sum_{n=1}^{N} e^{2\pi i a_{n}\theta}\Big)| \leq 2, \quad \forall\; N\geq 1, 
$$
Fatou's lemma implies that 
\begin{align*}
\|\lim_{N\to \infty} ( (S_{N,\ba}\phi)\circ f-S_{N,\ba}\phi) \|_{L^{2}(\nu_{K})} 
&=\lim_{N\to \infty} \|(S_{N,\ba}\phi)\circ f-S_{N,\ba}\phi \|_{L^{2}(\nu_{K})}\\
&= \lim_{N\to \infty}  \|(e^{2\pi i\theta} -1)\Big( \frac{1}{N}\sum_{n=1}^{N} e^{2\pi i a_{n}\theta}\Big)\|_{L^{2}(m_{\phi})}\\
&=\| \lim_{N\to \infty} (e^{2\pi i\theta} -1)\Big( \frac{1}{N}\sum_{n=1}^{N} e^{2\pi i a_{n}\theta}\Big)\|_{L^{2}(m_{\phi})}.
\end{align*} 
Since $\ba$ is uniformly behaved in $\N$, for any $0\leq \theta<1$, 
$$
\lim_{N\to \infty} (e^{2\pi i\theta} -1)\Big( \frac{1}{N}\sum_{n=1}^{N} e^{2\pi i a_{n}\theta}\Big)=0.
$$
This implies that  
$$
\lim_{N\to \infty} ( (S_{N,\ba}\phi)\circ f (x)-S_{N,\ba}\phi (x) ) =0, \quad \nu_{K}\hbox{-a.e.} \;\; x\in K.
$$

From the second step in the proof, the family 
$$
\mathcal{F}=\{ (S_{N,\ba}\phi)\circ f-S_{N,\ba}\phi\}_{N\in {\N}}
$$ 
is precompact in $C(K)$. Every sequence in $\mathcal{F}$ has a convergent subsequence. Suppose 
$$
\{ (S_{N_{k},\ba}\phi)\circ f-S_{N_{k},\ba}\phi\}_{k\in {\N}}
$$ 
is a convergent subsequence in $C(K)$ with a limit $\psi\in C(K)$. Since $K$ is $f$-minimal and $\nu_{K}$ is $f$-invariant, $\nu_{K}$ supports on the whole $K$. Thus $\psi \equiv 0$ on $K$. Therefore, 
$$
\lim_{N\to \infty} ((S_{N,\ba}\phi)\circ f (x)-S_{N,\ba}\phi (x)) =0, \quad \forall\; x\in X,
$$
and the convergence is uniform on $K$.

For any $x\in K$, let $\delta_{x}$ be the $\delta$-measure on $X$, that is $\delta_{x} (A) =1$ if $x\in A$ and $\delta_{x} (A)=0$ if $x\not\in A$ for any subset $A$ of $K$. The push-forward measures of $\delta_{x}$ by $f^{n}$ is $f^{n}_{*} \delta_{x} =\delta_{f^{n}x}$. Consider 
$$
\nu_{N,x} =\frac{1}{N} \sum_{n=1}^{N} \delta_{f^{a_{n}}x},
$$
which is a probability Borel measure on $K$. The push-forward measure by $f$ is 
$$
f_{*} \nu_{N,x} = \frac{1}{N} \sum_{n=1}^{N} \delta_{f^{a_{n}+1}x}. 
$$
For any $\phi\in C(K)$, consider 
$$
<\phi, f_{*}\nu_{N,x}-\nu_{N,x}> = \int_{X} \phi d(f_{*}\nu_{N,x}-\; \nu_{N,x}) =
$$
$$
 \int_{X}  ( (S_{N}^{\ba}\phi)\circ f-S_{N}^{\ba}\phi) \; d\delta_{x} = S_{N}^{\ba}\phi\circ f(x)-S_{N}^{\ba}\phi (x).
$$
Since the space of all probability measures on $K$ is a weakly compact metrizable space, 
let $\{\nu_{N_{k},x}\}_{k\in {\N}}$ be a convergent subsequence with a weak limit $\nu_{x}$, then 
$$
<\phi,  f_{*}\nu_{x}-\nu_{x}> =\lim_{k\to \infty} <\phi, f_{*}\nu_{N_{k},x}-\nu_{N_{k},x}> = 
$$
$$
\lim_{k\to \infty} <(S_{N_{k}}^{\ba}\phi)\circ f-S_{N_{k}}^{\ba}\phi, \delta_{x}>=
\lim_{k\to \infty} \Big( S_{N_{k}}^{\ba}\phi\circ f(x)-S_{N_{k}}^{\ba}\phi (x)\Big)=0.
$$
This implies that $f_{*}\nu_{x}=\nu_{x}$, that is, $\nu_{x}$ is an $f$-invariant probability measure on $K$. Since $f: K\to K$ is uniquely ergodic, we have that $\nu_{x}=\nu_{K}$. This implies that 
$$
\lim_{N\to \infty} \nu_{N,x} =\lim_{N\to \infty} \frac{1}{N} \sum_{n=1}^{N} \delta_{f^{a_{n}} x}=\nu_{K}
$$
for all $x\in X$.

Now for any $\phi\in C(K)$ and for any $x\in K$, 
$$
\lim_{N\to \infty} S_{N}^{\ba}\phi (x)= \lim_{N\to \infty} \frac{1}{N} \sum_{n=1}^{N} \phi (f^{a_{n}} x) = \lim_{N\to \infty} <\frac{1}{N} \sum_{n=1}^{N} \phi \circ f^{a_{n}}, \delta_{x}> 
$$
$$
= \lim_{N\to \infty} <\phi, \nu_{N,x}> =<\phi, \nu_{k}>=\int_{K}\phi \;d\nu =L,
$$
which completes the proof of the final step.

All three steps give us a proof of Theorem~\ref{main}.
\end{proof}

\section{Examples of $\ba$-MMLS and $\ba$-MMA Dynamical Systems}~\label{ammlsmma}

We say a dynamical system $f: X\to X$ is equicontinuous if for any $\epsilon >0$, there exists $\delta >0$ such that for any $x,y\in X$ with $d(x,y)<\delta$,
$$
d (f^{n} (x), f^{n}(y)) <\epsilon, \quad \forall\; n\in \N.
$$
A minimal and equicontinuous dynamical system $f$ is uniquely ergodic. A non-minimal but equicontinuous dynamical system $f$ can be decomposed into minimal subsystems. 


\medskip\begin{remark}
It is known that a surjective $f: X\to X$ is equicontinuous if and only if there exists a compatible metric $\widetilde{d} (\cdot, \cdot)$ on $X$ such that $f$ is an isometry concerning $\widetilde{d} (\cdot, \cdot)$, that is, $\widetilde{d} (f(x),f(y))=\widetilde{d} (x, y)$ for all $ x, y \in X$. Thus, if $f$ is surjective and equicontinuous, then $f$ must be a homeomorphism. Moreover, it is known that if $f$ is a homeomorphism and equicontinuous, then $f$ is minimal on the omega limit set $\omega(x)$. An equicontinuous transitive homeomorphism $f$ is always conjugate to a minimal rotation on a compact abelian metric group $K$, and the unique ergodic measure $\nu_{K}$ is the Haar probability measure on $K$. 
The measurable dynamical system $(K,f,\nu_{K})$ has a discrete spectrum and zero topological entropy. Thus, an equicontinuous dynamical system is always MMLS and MMA.  
\end{remark}

We have the following trivial (but still important) examples of $\ba$-MMLS dynamical systems.

\medskip
\begin{proposition}
Any equicontinuous dynamical system $f: X\to X$ is $\ba$-MMLS for any sequence $\ba$ taking values in $\N$.     
\end{proposition}

One consequence of our main theorem (Theorem~\ref{main}) is the following corollary.

\medskip
\begin{corollary}~\label{coro1}
Suppose $\ba$ is a uniformly behaved in $\N$ sequence and suppose $f: X\to X$ is an equicontinuous and $\ba$-MMA dynamical system. Then, for any $\phi \in C(X)$, 
$$
\lim_{N \to \infty} S_N^{\mathbf{a}} \phi(x) = \int_K \phi \; d\nu_K, \quad \forall\; x\in X,
$$
where $K\subset X$ is the $f$-minimal set such that $x\in \mbox{\rm Basin}_{\ba}(K)$ and
\( \nu_K \) is the unique Haar probability measure on $K$. 
\end{corollary}

We prove the following proposition to present some non-trivial examples of $\ba$-MMLS and $\ba$-MMA dynamical systems. We say a sequence $\ba$ is a subsequence of $\N$ if it is a strictly increasing sequence. The lower density of $\ba$ can be calculated as 
$$
\underline{D} (\ba) =\liminf_{N\to \infty} \frac{\# \{ a_{n}\mid 1\leq a_{n}\leq N\}}{N}. 
$$
 
\medskip
\begin{proposition}~\label{dad}
Suppose $\ba$ is a subsequence of $\mathbb{N}$. If $\underline{D} (\ba) >0$. Then for any subset $E\subset \mathbb{N}$,  
$$
\overline{D}_{\ba} (E)\leq \frac{\overline{D} (E)}{\underline{D}(\ba)}.
$$
\end{proposition}

\begin{proof}
Consider 
$$
\frac{\# \{ n\mid a_{n}\in E, \;1\leq n\leq N\}}{N}.
$$   
Since $\ba$ is a subsequence of $\mathbb{N}$ and consider it as a subset of $\N$, we have 
\begin{align*}
\frac{\# \{ n\mid a_{n}\in E,\; 1\leq n\leq N\}}{N}
&=\frac{\# ( \ba\cap E \cap [1, a_{N}])}{N}\\
&\leq  \frac{\# (E\cap [1, a_{N}])}{N} \\
&=\frac{\# (E \cap[1,a_{N}])}{a_N} \frac{a_N}{N}.
\end{align*}
By taking the upper limit, we get
\begin{align*}
\overline{D}_{\bf a}(E) 
&=\limsup_{N\to \infty} \frac{\# \{n\mid a_{n}\in E, \; 1\leq n\leq N\}}{N} \\
&\leq \limsup_{N\to \infty} \frac{\# (E \cap [1, a_{N}])}{a_N} \limsup_{N\to \infty} \frac{a_N}{N}\\
&= \frac{\limsup_{N\to \infty} \frac{\# (E\cap [1, a_{N}])}{a_N}}{\liminf_{N\to \infty} \frac{N}{a_N}}\\
&=\frac{\overline{D}(E)}{\underline{D}(\ba)}.
\end{align*}    
\end{proof}

From this proposition, we have that 

\medskip
\begin{theorem}~\label{mls}
Suppose $f$ is MLS. Then $f$ is $\ba$-MLS for any subsequence $\ba$ of $\N$ with $\underline{D}(\ba) >0$. 
\end{theorem}

\begin{proof}
Since $f$ is MLS, for any $\epsilon>0$, 
there exists $\delta>0$ such that for any $x, y\in X$ 
with $d(x,y)<\delta$, we have a subset $E\subset \N$ 
with $\overline{D}(E)<\epsilon \underline{D}(\ba)$ and 
$$
d(f^{n}x,f^{n}y)<\epsilon, \quad \forall\;  n\in \N\setminus E.
$$
From Proposition~\ref{dad}, $\overline{D}_{\ba}(E) <\epsilon$ and 
$$
d(f^{a_n}x,f^{a_n}y)<\epsilon, \quad \forall\;  n\in \N \hbox{ such that } a_{n}\not\in E.
$$
This says that $f$ is $\ba$-MLS.  
\end{proof}

The first non-trivial example of an $\ba$-MMLA and $\ba$-MMA dynamical system is a Denjoy counterexample in the setting of circle homeomorphisms. 

Remember that $\T=\{ z\in \C \;|\; |z|=1\}$ is the unit circle in the complex plane $\C$. If a continuous dynamical system $f: \T\to \T$ is equicontinuous, then 
\begin{itemize}
\item[(1)] $\deg (f)=0$ and for every $z\in \T$, $f^{n}z\to E$ as $n\to \infty$, where $E$ is the fixed point set of $f^{2}$,
\item[(2)] $\deg(f)=1$ and $f$ is topologically conjugate to a rigid rotation $R_{\rho}(z) =e^{2\pi i \rho} z$ for some $0\leq \rho <1$, or
\item[(3)] $\deg(f)=-1$ and $f^{2}$ is topologically conjugate to the identity
\end{itemize}
(see~\cite{CMY, Va}). 
Suppose $f: \T\to \T$ is a homeomorphism. If $f$ is orientation-reversing, it has two fixed points and $f^{2}$ is orientation-preserving and conjugate to the identity. Thus, 
it is equicontinuous. Therefore, we only consider an orientation-preserving homeomorphism $f$. It is semi-conjugate to a rigid rotation $R_{\rho}(z) =e^{2\pi i \rho} z: \T\to \T$ for some $0\leq \rho <1$ due a theory developed by Poincar\'e (see, for example,~\cite[Chapter 1]{J} or~\cite{Ma}). That is, we have a continuous surjective map $h: \T\to \T$ such that 
$$
h\circ f=R_{\rho}\circ h.
$$ 
We call $0\leq \rho <1$ the rotation number of $f$. If the rotation number $\rho$ is rational, then $f$ is equicontinuous. 

Now we assume the rotation number $\rho$ is an irrational number. If $h$ is a homeomorphism, then $f$ is equicontinuous since the rigid rotation is equicontinuous. If $h$ is not a homeomorphism, we call $f$ a Denjoy counterexample since Denjoy was the first person to show that such a circle homeomorphism exists. 

For $x\in \T$, we say $x$ is non-wandering if for any neighborhood $U$ of $x$, there is an integer $n>0$ such that $f^{n}(U) \cap U\not= \emptyset$. Let $\Omega_{f}$ be the set of all non-wandering points. Then $h$ is a homeomorphism if and only if $\Omega_{f}=\T$. When $f$ is a Denjoy counterexample, $\Omega_{f}$ is a Cantor set. 

For any $x\in \T$, the omega limit set $\omega (x)$ is defined as
$$
\omega (x) =\cap_{n=1}^{\infty} \overline{\{ f^{k}(x)\;|\; k\geq n\}}.
$$
It is known in the setting of circle homeomorphisms that
$$
\omega (x) =\Omega_{f}, \quad \forall\; x\in \T.
$$
Thus, $\Omega_{f}$ is $f$-minimal and, for any subsequence $\ba$ of natural numbers, 
$$
\hbox{Basin}_{\ba} (\Omega_{f}) =\T,
$$ 
which implies that $f$ is $\ba$-MMA. 

In~\cite{FJ}, we proved the following result.

\medskip
\begin{theorem}
Every Denjoy counterexample $f: \T\to \T$ is MMLS and MMA but not equicontinuous. 
\end{theorem}

Then Proposition~\ref{dad} implies that 

\medskip
\begin{theorem}~\label{den}
Suppose $\ba$ is a subsequence of $\N$ with $\underline{D}(\ba) > 0$. Then Every Denjoy counterexample $f: \T\to \T$ is $\ba$-MMLS and $\ba$-MMA but not equicontinuous. 
\end{theorem}

Since $f$ is semi-conjugate to the rigid rotation $R_{\rho}$ and since $R_{\rho}$ is uniquely ergodic, $f|\Omega_{f}$ is uniquely ergodic. Another way to argue that $f|\Omega_{f}$ is uniquely ergodic is to use the facts that $f|\Omega_{f}$ is minimal and MLS (refer to~\cite{Fo}). As a consequence of our main theorem (Theorem~\ref{main}), we have that

\medskip
\begin{corollary}~\label{coro2}
Suppose $\ba$ is a subsequence of $\N$ with $\underline{D}(\ba)>0$ and uniformly behaved in $\N$, and suppose $f$ is a Denjoy counterexample. Then for every $\phi \in C(\T)$,
$$
\lim_{N \to \infty} S_N^{\mathbf{a}} \phi(x) = \int_{\Omega_{f}} \phi \; d\nu, \quad \forall \; x \in \T,
$$
where $\nu$ is the unique $(f|\Omega_{f})$-invariant probability measure.
\end{corollary}


 
The second non-trivial example of an $\ba$-MMLA and $\ba$-MMA dynamical system is an interval map with zero topological entropy.

Let $I=[a,b]$ be a finite interval in the real line $\R$. Suppose $f: I\to I$ is a continuous map with zero topological entropy. Then the period of any periodic point of $f$ is $2^{n}$ due to Sarkovskii's theorem. For any $x\in I$, let
$$
\omega (x) = \cap_{n=0}^{\infty} \overline{\{ f^{k}x \;|\; k\geq n\}}.
$$
be the $\omega$-limit set of $x$ under $f$.  Let $K=\omega (x)$ and consider $f: K\to K$, which is minimal.  We have two cases: 
\begin{itemize}
\item[(I)]  the set $K$ is finite and 
\item[(II)] the set $K$ is infinite. 
\end{itemize}

In Case (I), $K =\{ p_{0}, p_{1}, \cdots, p_{2^{n}-1}\}$ is a periodic orbit and $f: K\to K$ is equicontinuous.

In Case (II), since the topological entropy of $f$ is zero,  $K$ contains no periodic point, and we can find a sequence of sets, 
$$
\eta_{n}=\{ I_{n,k}\;|\; 0\leq k\leq 2^{n}-1\},\;\; n\in\mathbb{N},
$$
of closed intervals $I_{n,k}$ such that 
\begin{itemize}
\item[1)] for each $n\geq 1$, intervals in $\eta_{n}$ are pairwise disjoint;
\item[2)] $f(I_{n, k}) =I_{n, k+1 \pmod{2^{n}}}$ for any $n\geq 1$ and $0\leq k\leq 2^{n}-1$;
\item[3)] for each $n\geq 1$ and $I_{n, k}\in \eta_{n}$, two and only two intervals $I_{n+1, k}$ and $I_{n+1, k+2^{n}}$ in $\eta_{n+1}$ are sub-intervals of $I_{n, k}$, that is,
$$
I_{n+1, k}\cup I_{n+1, k+2^{n}}\subseteq I_{n, k};
$$
\item[4)] $K\subseteq \cap_{n=1}^{\infty} \cup_{k=0}^{2^{n}-1} I_{n,k}$.
\end{itemize}
(See, for example,~\cite[Proposition 5.21]{R}.) 
We put a symbolic coding on each interval $I_{n,k}$ in $\eta_{n}$ for $n\geq 1$ and $0\leq k\leq 2^{n}-1$.
Consider the binary expansion of $k$ as
$$
k= i_{0} + i_{1} 2+ \cdots i_{m}2^{m} +\cdots +i_{n-1}2^{n-1}
$$
where $i_{0}, i_{1}, \cdots, i_{m}, \cdots, i_{n-1} =0$ or $1$. 
Let us only remember the coding $w_{n}=i_{0}i_{1}\cdots i_{m}\cdots i_{n-1}$ and label $I_{n, k}$ by $w_{n}$, that is, $I_{n, k}=I_{w_{n}}$. 

Define the space of codings,
$$
\Sigma_{n} =\{ w_{n} =i_{0}i_{1}\cdots i_{m}\cdots i_{n-1}\;|\; i_{m}\in \{0,1\},\; m=0, 1, \cdots, n-1\},
$$
with the product topology.  
We have a natural shift map,
$$
\sigma_{n}(w_{n}) =i_{1}\cdots i_{m}\cdots i_{n-1}: \Sigma_{n}\to \Sigma_{n-1}.
$$
It is a continuous $2$-to-$1$ map. Then we have an inverse limit system,
$$
\{ (\sigma_{n}: \Sigma_{n}\to \Sigma_{n-1})\;; \; n\in \mathbb{N}\}.
$$
Let 
$$
(\sigma: \Sigma \to \Sigma) =\lim_{\longleftarrow} \;(\sigma_{n}: \Sigma_{n}\to \Sigma_{n-1}).
$$
Explicitly, we have that
$$
\Sigma=\{ w =i_{0}i_{1}\cdots i_{n-1}\cdots \;|\; i_{n-1}\in \{0,1\},\; n\in \mathbb{N} \}
$$
and
$$
\sigma: w=i_{0}i_{1}\cdots i_{n-1}\cdots \to \sigma(w) =i_{1}\cdots i_{n-1}\cdots. 
$$
Consider the metric $d(\cdot, \cdot)$ on $\Sigma$ defined as 
$$
d (w, w') =\sum_{n=1}^{\infty} \frac{|i_{n-1}-i_{n-1}'|}{2^{n}}
$$
for $w=i_{0}i_{1}\cdots i_{n-1}\cdots$ and $w'=i_{0}'i_{1}'\cdots i_{n-1}'\cdots$. 
It induces the same topology as the topology from the inverse limit and makes $\Sigma$ a compact metric space. 

On $\Sigma$, we have a map called the {\em adding machine} denoted as $add$ and defined as follows: for any $w=i_{0}i_{1}\cdots \in \Sigma$, consider the formal binary expansion 
$a=\sum_{n=1}^{\infty} i_{n-1}2^{n-1}$, then $a+1= \sum_{n=1}^{\infty} i_{n-1}'2^{n-1}$ has a unique formal binary expansion. 
Here ``a+1'' is just a notation. It means if $i_{0}=0$, then $i_{0}'=1$ and all other $i_{n}'=i_{n}$ and if $i_{0}=1$, then $i_{0}'=0$ and then consider $i_{1}+1$, and so on.
The adding machine is 
$$
add (w) =i_{0}'i_{1}'\cdots i_{n-1}'\cdots: \Sigma\to \Sigma.
$$
It is a homeomorphism of $\Sigma$ to itself.
 
For a point $w=i_{0}i_{1}\cdots i_{n-1}\cdots \in \Sigma$, let $w_{n}=i_{0}i_{1}\cdots i_{n-1}$. 
We use 
$$
[w]_{n}=\{ w'\in \Sigma \;|\; w_{n}'=w_{n}\}
$$
to denote the $n$-cylinder containing $w$. 
One can check that for every $w\in \Sigma$, 
$$
\cdots \subset I_{w_{n}} \subset I_{w_{n-1}}\subset \cdots \subset I_{w_{2}}\subset I_{w_{1}},
$$
is a nested sequence of closed intervals. Therefore
$$
I_{w} =\cap_{n=1}^{\infty} I_{w_{n}}
$$
is a non-empty, connected, and compact subset of $I$. Moreover, we have that $f(I_{w}) =I_{add(w)}$ and
$$
K\subseteq \cup_{w\in \Sigma} I_{w}.
$$ 

The set $\mathcal{C}=\{ [w]_{n} \;| \; w\in \Sigma, n\in \mathbb{N}\}$ of all $n$-cylinders forms a topological basis for $\Sigma$. Let $\mathcal{B}$ be the $\sigma$-algebra generated $\mathcal{C}$. We have a standard probability measure $\nu$ defined on $\mathcal{B}$ such that $\nu ([w]_{n}) =1/2^{n}$ for all $w\in \Sigma$. It is $add$-invariant in the meaning that 
$$
\nu (add^{-1}(A))=\nu(A), \quad \forall A\in\mathcal{B}.
$$ 
Moreover, $\nu$ is the unique $add$-invariant probability measure. Therefore, we have that    

\medskip
\begin{proposition}~\label{sca}
The dynamical system $add: \Sigma\to \Sigma$ is minimal, uniquely ergodic, and equicontinuous.
\end{proposition}
 
In Case (II), we consider two different cases:
\begin{itemize}
\item[(a)] there is a point $w_{0}\in \Sigma$ such that $I_{w_{0}}=\{x_{w_{0}}\}$ contains only one point or 
\item[(b)] all $I_{w}$ for $w\in \Sigma$ are closed intervals.
\end{itemize}

For Case (II) $(a)$, we have that $\omega (x)=\omega (x_{w_{0}})$.   Without loss of generality, we assume $x=x_{w_{0}}$ and every $I_{w}=\{x_{w}\}$ contains only one point. Let 
$$
\tau_{n}=\max_{w_{n}\in \Sigma_{n}} |I_{w_{n}}|.
$$ 
Then we have that $\tau_{n}\to 0$ as $n\to \infty$. This implies that the map $h(x_{w_n}) =w_{n}: K\to \Sigma$ is a homeomorphism such that 
$$
f=h^{-1} \circ add\circ h.
$$ 
From Proposition~\ref{sca}, we have that    

\medskip
\begin{proposition}~\label{IIa}
In Case II $(a)$, $f: K\to K$ is minimal, uniquely ergodic, and equicontinuous.
\end{proposition} 

From Corollary~\ref{coro1}, we have that 

\medskip
\begin{corollary}~\label{coroze}
Suppose $\ba$ is a uniformly behaved in $\N$ sequence and suppose $f: I\to I$ is an interval map having zero topological entropy in Case (I) or Case (II) (a) and $\ba$-MMA. Then, for any $\phi \in C(I)$, 
$$
\lim_{N \to \infty} S_N^{\mathbf{a}} \phi(x) = \int_K \phi \; d\nu_K, \quad \forall\; x\in X,
$$
where $K\subset I$ is the $f$-minimal set such that $x\in \mbox{\rm Basin}_{\ba}(K)$ and
\( \nu_K \) is the unique \( (f|K) \)-invariant probability measure.
\end{corollary}

For Case (II) $(b)$, the map $h(I_{w_n}) =w_{n}: \cup_{w\in \Sigma} I_{w} \to \Sigma$ is a continuous map but not a homeomorphism such that 
$$
h\circ f=add\circ h.
$$ 
In~\cite{JNon}, we proved the following theorem.

\medskip
\begin{theorem}~\label{IIb}
In Case (II) $(b)$, $f: K\to K$ is minimal, MLS, uniquely ergodic, but not equicontinuous.  
\end{theorem}

A point $x\in I$ is called non-wandering if for any neigborhood $U$ of $x$, there is an integer $n>0$ such that $f^{n}(U)\cap U\not= \emptyset$. Let $\Omega_{f}$ be the set of all non-wandering points. For the continuous interval map $f: I\to I$ with zero topological entropy, from the above analysis, we have two cases: (1)
$
\Omega_{f} =P_{f}
$
and (2) 
$
\Omega_{f}=C_{f}\cup P_{f},
$
where $C_{f}$ is a Cantor set and $P_{f}$ is the set of periodic points of period $2^{n}$, $n\geq 0$. 

Suppose $\ba$ is a subsequence of natural numbers. In the first case (1), there is a periodic cycle $\alpha$ in $P_{f}$ such that 
$$
\hbox{Basin}_{\ba} (\alpha) = I\setminus (P_{f}\setminus \alpha).
$$
and, for any other periodic cycle $\beta\not=\alpha$ in $P_{f}$, 
$$
\hbox{Basin}_{\ba} (\beta) = \beta
$$
In the second case (2), 
$$
\hbox{Basin}_{\ba} (C_{f}) = I\setminus P_{f}.
$$
and, for any periodic cycle $\beta$ in $P_{f}$, 
$$
\hbox{Basin}_{\ba} (\beta) = \beta.
$$
This implies that $f$ is $\ba$-MMA for any subsequence $\ba$ of natural numbers. This, combined with Theoprem~\ref{IIb} and Proposition~\ref{dad}, implies the following theorem. 

\medskip
\begin{theorem}~\label{int}
Suppose $\ba$ is a subsequence of $\N$ with $\underline{D}(\ba)>0$. Then any interval map $f: I\to I$ having zero topological entropy in Case (II) $(b)$ is $\ba$-MMLS and $\ba$-MMA but not equicontinuous. 
\end{theorem}

On each periodic cycle $P$ in $P_{f}$, $f|P$ is uniquely ergodic. Since $f|C_{f}$ is topologically conjugate to the adding machine, $f|C_{f}$ is uniquely ergodic. Thus, $f$ is a minimally uniquely ergodic. 
We have the following corollary as another consequence of Theorem~\ref{main}.

\medskip
\begin{corollary}~\label{coro3}
Suppose $\ba$ is a subsequence of $\N$ with $\underline{D}(\ba)>0$ and uniformly behaved in $\N$, and suppose $f: I\to I$ is a continuous interval map with zero topological entropy in Case II $(b)$. 
Then for every $\phi \in C(I)$,
$$
\lim_{N \to \infty} S_N^{\mathbf{a}} \phi(x) = \int_K \phi \; d\nu_{K}, \quad \forall\; x \in I,
$$
where $K$ is the minimal set such that $x\in\mathrm{Basin}_{\mathbf{a}}(K)$ and $\nu_{K}$ is the unique $(f|K)$-invariant probability measure.
\end{corollary}



\section{\bf Examples of Uniformly Behaved in $\N$ Sequences}~\label{eubs}

Let $\Z$, $\Q$, and $\R$ be the sets of integers, rational numbers, and real numbers. A sequence $(x_n)_{n\in \mathbb{N}}$ of real numbers is uniformly distributed modulo $1$ if for all $0 \leq s < t \leq 1$,
$$
\lim_{N \to \infty} \frac{\#\{1 \leq n \leq N \mid s \leq x_n \bmod 1 \leq t\}}{N} = t - s.
$$
Given a sequence ${\bf z}=(z_{n})_{n\in \N}$ taking values in $\Z$. If, for every irrational numbers $t$,  $(t z_{n})_{n\in \N}$ is uniformly distributed modulo $1$, then we call ${\mathbf z}$ uniformly distributed modulo $1$. The Weyl criterion (refer to~\cite[Chapter 1, Theorem 2.2]{KN} says that ${\mathbf z}$ is uniformly distributed modulo $1$ is equivalent to
$$
\lim_{N\to \infty} \frac{1}{N} \sum_{n=1}^{N} e^{2\pi i t z_{n}} =0, \quad \forall\; t\in \R\setminus \Q. 
$$

The sequence $\mathbf{z}$ is said to be uniformly distributed in $\mathbb{Z}$ if for every $m \in \mathbb{N}$ and $0 \leq r < m$,
$$
\lim_{N \to \infty} \frac{\#\{1 \leq n \leq N \mid z_n \equiv r \pmod{m}\}}{N} = \frac{1}{m}.
$$
This property is equivalent to 
$$
\lim_{N\to \infty} \frac{1}{N} \sum_{n=1}^{N} e^{2\pi i t z_{n}} =0, \quad \forall\; t\in \Q\setminus \Z. 
$$
(Refer to \cite[Chapter 5, Corollary 1.1]{KN}).


A basic uniformly behaved in $\N$ sequence is the sequence $(n)_{n\in \N}$ of all natural numbers. 

A classical but non-trivial example of a uniformly behaved in $\N$ sequence is given by the prime Omega function $\Omega (n)$ defined as 
$$
\Omega (n) =k_{1} +\cdots +k_{r},
$$
where $n=p_{1}^{k_{1}}\cdots p_{r}^{k_{r}}$ is the prime factor decomposition. Consider the sequence ${\bf \Omega}=(\Omega (n))_{n\in \N}$. 
From~\cite{P,Se}, it is known that ${\bf \Omega}$ is uniformly distributed in $\Z$ and, from~\cite{D}, it is known that ${\bf \Omega}$ is uniformly distributed modulo $1$. Thus, ${\bf \Omega}$ is uniformly behaved in $\N$. There is an interesting proof of this example from the point of view of dynamical systems recently in~\cite{BR,R}. 
%

In this section, we will construct other new examples of uniformly behaved sequences. In~\cite{FJ}, we have defined the linear disjointness between a sequence of complex numbers and a dynamical system.
Let $\bc =(c_{n})_{n\in \N}$ be a sequence of complex numbers. 

\medskip
\begin{definition}~\label{ld}
 We say $\bc$ is linearly disjoint from a continuous dynamical system $f:X\to X$ if for any $\phi\in C(X)$, 
$$
\lim_{N\to \infty} \frac{1}{N} \sum_{n=1}^{N} c_{n} \phi (f^{n}x) =0, \quad \forall\; x\in X. 
$$
\end{definition}

Suppose $A$ is a subset of $\N$ with density $D(A)>0$. Its indicator $t_{n}$ is defined as $ t_{n}=1$ if $n\in A$ and $t_{n}=0$ otherwise. Let 
\begin{equation}~\label{ac}
c_{n} =2 t_{n}-1, \quad n\in \N.
\end{equation}
Then $\bc$ is a sequence taking values in $\{-1, 1\}$. 
List $A$ as a subsequence $\ba=(a_{n})_{n\in \N}$ of $\N$ (we will not distinguish $A$ and $\ba$ for a subsequence of $\N$ in the rest of the paper). 
For such two sequences $\ba$ and $\bc$, we have that 

\medskip
\begin{theorem}~\label{ldt}
Suppose $f: X \to X$ is uniquely ergodic and suppose $\nu$ is the unique $f$-invariant probability measure. Then  
$\bc$ is linearly disjoint from $f$ if and only if 
for any $\phi\in C(X)$,
\begin{equation}~\label{ldte}
\lim_{N\to \infty}  \frac{1}{N} \sum_{n=1}^{N} \phi (f^{a_{n}}x)  = \frac{1}{2D(A)} \int_{X} \phi \; d\nu, \quad \forall\; x\in X.
\end{equation}
\end{theorem}

\begin{proof}
Since $\#(A\cap [1,a_{N}]) =N$, for any $\phi\in C(X)$,
\begin{align*}
\frac{N}{a_{N}} \frac{1}{N} \sum_{n=1}^{N} \phi (f^{a_{n}}x) 
&=\frac{1}{a_{N}} \sum_{n=1}^{a_{N}} t_{n} \phi (f^{n}x) \\
&=
\frac{1}{2} \frac{1}{a_{N}} \sum_{n=1}^{a_{N}} c_{n} \phi (f^nx) +\frac{1}{2} \frac{1}{a_{N}} \sum_{n=1}^{a_{N}} \phi (f^{n}x).
\end{align*}

Since $f$ is uniquely ergodic, from Theorem~\ref{uet}, 
$$
\lim_{N\to \infty} \frac{1}{a_{N}} \sum_{n=1}^{a_{N}} \phi (f^{n}x)=\int_{X} \phi \; d\nu, \quad \forall\; x\in X.
$$

Suppose $\bc$ is linear disjoint from $f$, then
$$
\lim_{N\to \infty} \frac{1}{a_{N}} \sum_{n=1}^{a_{N}} c_{n} \phi (f^nx)=0, \quad \forall \; x\in X. 
$$
By the definition,
$$
D(A)=\lim_{N\to \infty} \frac{N}{A_{N}}.
$$
We get that 
$$
D(A) \lim_{N\to \infty} \frac{1}{N} \sum_{n=1}^{N} \phi (f^{a_{n}}x) = 0+ \frac{1}{2} \int_{X} \phi \; d\nu.
$$
Thus,
$$
\lim_{N\to \infty} \frac{1}{N} \sum_{n=1}^{N} \phi (f^{a_{n}}x) = \frac{1}{2D(A)} \int_{X} \phi \; d\nu.
$$

On the other hand, if $(\ref{ldte})$ holds, then we have that 
$$
\lim_{N\to \infty} \frac{1}{a_{N}} \sum_{n=1}^{a_{N}} c_{n} \phi (f^nx)=0, \quad \forall \; x\in X,
$$
which says that $\bc$ is linearly disjoint from $f$.
\end{proof}

\begin{corollary}
Suppose $\ba$ is a subsequence of $\N$ with $D(\ba)>0$ and suppose $\bc$ is linearly disjoint from an $\ba$-MMLS, $\ba$-MMA, and minimal uniquely ergodic dynamical system. If $\ba$ is uniformly behaved in $\N$, then $D(\ba)=1/2$.
\end{corollary}

In~\cite{FJ,JPro}, we have the following definition. Suppose $\bc$ is a sequence of complex numbers satisfying the technical condition: there is a real number $\lambda >1$ such that
$$
\limsup_{N\to \infty} \frac{1}{N} \sum_{n=1}^{N} |c_{n}|^{\lambda} < \infty.
$$

\medskip
\begin{definition}~\label{oss}
We say $\bc$ is an oscillating sequence if for any $0\leq \theta <1$,
$$
\lim_{N\to \infty} \frac{1}{N} \sum_{n=1}^{N} c_{n} e^{2\pi i n \theta} =0.
$$
\end{definition}

We proved in~\cite{FJ} that 

\medskip
\begin{thma}~\label{fjt}
Any oscillating sequence is linearly disjoint from any MMLS and MMA dynamical system $f:X\to X$.
\end{thma}
 
The following result connects a uniformly behaved in $\N$ sequence $\ba$ and the oscillation property for the sequence $\bc$ defined in (\ref{ac}) with $D(\ba)>0$. 

\medskip
\begin{theorem}~\label{ubos}
 The sequence $\ba$ is uniformly behaved in $\N$ and $D(\ba)=1/2$ if and only if the sequence 
$\bc=(c_{n})_{n\in \N}$ is oscillating. 
\end{theorem} 

\begin{proof}
Suppose $\bc$ is oscillating. Similar to the proof of Theorem~\ref{ldt}, we have that for any $0<\theta <1$,
\begin{align*}
\frac{N}{a_{N}} \frac{1}{N} \sum_{n=1}^{N} e^{2\pi i a_{n} \theta} 
&= \frac{1}{a_{N}}\sum_{n=1}^{a_{N}} t_{n} e^{2\pi i n \theta}\\
&=\frac{1}{2} \frac{1}{a_{N}} \sum_{n=1}^{a_{N}} c_{n} e^{2\pi i n \theta} +\frac{1}{2} \frac{1}{a_{N}}\sum_{n=1}^{a_{N}} e^{2\pi i n \theta}\\
& =\frac{1}{2} \frac{1}{a_{N}} \sum_{n=1}^{a_{N}} c_{n} e^{2\pi i n \theta} +\frac{1}{2} \frac{1}{a_{N}} \frac{1-e^{2\pi i (a_{N}+1) \theta}}{1-e^{2\pi i \theta}}
\end{align*}
This implies that
\begin{align*}
\lim_{N\to \infty}\frac{1}{N} \sum_{n=1}^{N} e^{2\pi i a_{n} \theta}
&= \frac{1}{2D(A)} \Big(\lim_{N\to \infty} \frac{1}{a_{N}} \sum_{n=1}^{a_{N}} c_{n} e^{2\pi i n \theta} +\lim_{N\to \infty} \frac{1}{a_{N}} \frac{1-e^{2\pi i (a_{N}+1) \theta}}{1-e^{2\pi i \theta}}\Big) \\
&=\frac{1}{2D(A)} (0+0) =0, 
\end{align*}
which says that $\ba$ is uniformly behaved in $\N$. 

Since 
$$
\lim_{N\to \infty} \frac{1}{N} \sum_{n=1}^{N} c_{n}=0
$$
and $c_{n}=2t_{n}-1$,
$$
\lim_{N\to \infty} \frac{1}{N} \sum_{n=1}^{N} t_{n} =\frac{1}{2}.
$$
But
$$
D(\ba)= \lim_{N\to \infty} \frac{\#(A\cap [1, N])}{N} =\lim_{N\to \infty} \frac{1}{N} \sum_{n=1}^{N} t_{n}.
$$
We get
$$
D(\ba)=\frac{1}{2}.
$$
We proved the ``only if'' part of the theorem. 

Suppose $\ba$ is uniformly behaved in $\N$. A similar argument from the above implies that for $0<\theta <1$,  
$$
\frac{1}{N} \sum_{n=1}^{N} c_{n} e^{2\pi i n \theta}=0.
$$
The assumption $D(A) =1/2$ implies that for $\theta=0$, 
$$
\frac{1}{N} \sum_{n=1}^{N} c_{n} =0.
$$
Thus, $\bc$ is oscillating.
This completes the proof of the ``if'' part of the theorem. 
\end{proof}

Theorem~\ref{ubos} provides many uniformly behaved in $\N$ sequences. Consider the Thue-Morse sequence,
$$
\omega_{tm}=01101001\cdots=t_{0}t_{1}t_{2}t_{3}t_{4}t_{5}t_{6}t_{7}\cdots.
$$
It can be defined as follows. Let $s(n)$ be the number of $1$'s in the binary expansion of $n$. Then $t(n) = 0$ if $s(n)$ is even and $t(n)=1$ if $s(n)$ is odd. 
Let 
$$
TM=\{ n\in \N \mid t(n)=1\}.
$$ 
Then $t(n)$ is the indicator of $TM$. List $TM$ as a subsequence ${\bf tm}= (tm_{n})_{n\in \N}$ of $\N$. Let $\bc=(c_{n}=2t(n)-1)_{n\in \N}$ be defined in (\ref{ac}), we have that (see~\cite{Gel})
$$
\sup_{0\leq \theta<1} |\sum_{n=1}^{N} c_{n} e^{2\pi i n\theta} | =O(N^{\frac{\log 3}{\log 4}}).
$$
Thus $\bc$ is an oscillating sequence. Theorem~\ref{ubos} implies that

\medskip
\begin{example}~\label{tms}
 The sequence ${\bf tm}$ is uniformly behaved in $\N$. 
\end{example}

Let $f_{11}(n)$ denote the number of times the pattern $11$ appears in the binary expansion of $n$. Define $r(n) =0$ if $f_{11}(n)$ is odd and $r(n) =1$ if $f_{11}(n)$ is even. The sequence 
$$
\omega_{rs}=r(0) r(1) \cdots 
$$ 
is called the Rudin-Shapiro sequence. Let 
$$
RS=\{ n\in \N \mid r(n)=1\}.
$$
Then $r_{n}$ is the indicator of $RS$. List $RS$ as a subsequence ${\bf rs}= (rs_{n})_{n\in \N}$ of $\N$. Let $\bc =(c_{n}=2r(n)-1)_{n\in \N}$ be defined in (\ref{ac}). We have that 
(see~\cite[78-79]{AS}),  
$$
\sup_{0\leq \theta<1} |\sum_{n=1}^{N} c_{n} e^{2\pi i n\theta} | \leq C\sqrt{N}.
$$
This implies that $\bc$ is an oscillating sequence. Theorem~\ref{ubos} implies that 

\medskip
\begin{example}~\label{rss}
 The sequence ${\bf rs}$ is uniformly behaved in $\N$.
\end{example}

\medskip
\begin{remark}~\label{qu}
Examples~\ref{tms} and~\ref{rss} provide a way to find more uniformly behaved in $\N$ sequences. For example, any automatic sequence that does not correlate with periodic sequences can be used to find a uniformly behaved in $\N$ sequence since they are highly Gowers uniform (see~\cite{BKM}), thus, are not only oscillating sequences but also higher order oscillating sequences as defined in~\cite{JPro}. These examples of uniformly behaved in $\N$ sequences have a common property: their densities are $1/2$. An interesting problem for us is to find a uniformly behaved sequence in $\N$ which is a subsequence of $\N$ and whose density is not $1/2$, in particular, whose density is $0$. Another problem is to find a non-increasing uniformly behaved in $\N$ sequence other than ${\bf \Omega}$.
\end{remark}

The Liouville function $\lambda (n)$ is defined as
$$
\lambda (n) =(-1)^{\Omega (n)}, \quad n\in \N. 
$$ 
It takes values in $\{-1,1\}$. A classical result from Davenport's estimate (see~\cite{Dev}) implies that there is a constant $c>0$ such that 
$$
\sum_{n=1}^{N} \lambda(n) e^{2\pi i n\theta} =O\Big( N \exp{\Big(-c\frac{\log N}{\log\log N}\Big)} \Big), \quad \hbox{uniformly on $0\leq \theta<1$}.
$$
Thus, ${\bf l}=(\lambda (n))_{n\in \N}$ is an oscillating sequence. 

Let $EF$ denote the set of all natural numbers $n$ such that  $\Omega (n)$ is even. Then the indicator of $EF$ is
$$
t_{n} =\frac{\lambda (n)+1}{2}.
$$
List $EF$ as a subsequence ${\bf ef}= (ef_{n})_{n\in \N}$ of $\N$.
Then Theorem~\ref{ubos} implies that

\medskip
\begin{example}~\label{ef}
 The sequence ${\bf ef}$ is uniformly behaved in $\N$.
\end{example}

Similarly,  
$$
\widetilde{t}_{n} =\frac{1-\lambda (n)}{2}
$$
is the indicator of $OF$, the set of natural numbers $n$ 
such that $\Omega (n)$ is odd. List $OF$ as a subsequence 
${\bf of}= (of_{n})_{n\in \N}$ of $\N$. Then Theorem~\ref{ubos} gives another example.  

\medskip
\begin{example}~\label{of}
The sequence ${\bf of}$ is uniformly behaved in $\N$.
\end{example}

\medskip
\begin{remark}
Consider $a_{n}=P(n)$ where $P\in \Z[n]$ is a polynomial with integer coefficients. The sequence $\ba=(a_{n})_{n\in \N}$ is uniformly distributed modulo $1$ but not uniformly distributed in $\Z$ if the degree of $P$ is greater than $1$. Therefore, $\ba$ is not uniformly behaved in $\N$. Such a sequence is useful in studying the convergence of $S_{N}^{\ba} \phi$ almost everywhere (see~\cite{B}), but not useful in studying the convergence of $S_{N}^{\ba} \phi$ for every point. This is the main point for us to use Definition~\ref{ubs}. We will construct more examples, which are uniformly distributed modulo $1$ but not uniformly behaved in $\N$, in the next section. 
\end{remark}

In addition to the two problems in Remark~\ref{qu}, we raise the following conjecture.

\medskip
\begin{conjecture}~\label{oden}
Any Denjoy counterexample in the setting of circle homeomorphisms is ${\bf \Omega}$-MMLA and ${\bf \Omega}$-MMA.
\end{conjecture}

\section{Square-Free Numbers And Uniform Behavior in $\N$}~\label{sfn}

We say a natural number $n$ is square-free if $n=p_{1}\cdots p_{r}$ for $r$ distinct prime numbers $p_{1}$, $\cdots$, $p_{r}$. In number theory, studying the structure of all square-free numbers is important. Let $\SF$ be the set of all square-free numbers.
Then we know the density,
$$
D(\SF) =\frac{6}{\pi^{2}}.
$$

The M\"obius function $\mu(n)$ is defined as $\mu (n) =\lambda (n)$ if $n\in \SF$ and $\mu (n)=0$ if $n\not\in \SF$. It takes values in $\{-1, 0, 1\}$ and $\mu^{2} (n)$ is the indicator of $\SF$. 
A classical result from Davenport's estimate (see~\cite{Dev}) implies that there is a constant $A>0$ such that
$$
\sum_{n=1}^{N} \mu (n) e^{2\pi i n\theta} =O\Big( \frac{N}{\log^{A}N}\Big), \quad \hbox{uniformly on $0\leq \theta<1$}.
$$
Thus, ${\bf u}=(\mu (n))_{n\in \N}$ is an oscillating sequence. 

We will use the identity 
$$
\mu^{2} (n) =\sum_{m^{2} | n} \mu (m)
$$ 
in the following proofs. 
Suppose $\ba$ is a sequence taking values in $\N$.

 \medskip
 \begin{definition}~\label{sfrd} 
We say $\ba$ is square-free uniformly distributed modulo $1$ if for all irrational numbers $0<\theta<1$, 
 $$
 \lim_{N\to \infty} \frac{1}{N} \sum_{1\leq n\leq N, n\in \SF}  e^{2\pi i a_{n}\theta} =0.
 $$
 \end{definition} 

\medskip
 \begin{definition}~\label{sfudz} 
 We say $\ba$ is square-free uniformly distributed in $\Z$ if for all rational numbers $0<\theta<1$, 
 $$
 \lim_{N\to \infty} \frac{1}{N} \sum_{1\leq n\leq N, n\in \SF} e^{2\pi i a_{n}\theta} =0.
 $$
 \end{definition}

\medskip
 \begin{definition}~\label{sfub} 
 We say $\ba$ is square-free uniformly behaved in $\N$ if for all numbers $0<\theta<1$, 
 $$
 \lim_{N\to \infty} \frac{1}{N} \sum_{1\leq n \leq N, n\in \SF} e^{2\pi i a_{n}\theta} =0.
 $$
 \end{definition}

  \medskip
 \begin{definition}~\label{ca}
 We say $\ba$ is completely additive if $a_{mn} =a_{n} +a_{m}$ for all $n, m\in \N$.  
 \end{definition}

    The sequence ${\bf n} =(n)_{n\in \N}$ is an example of a completely additive sequence. Another example of a completely additive sequence is ${\bf \Omega}$. 

\medskip
\begin{theorem}~\label{sfubt}
Suppose $\ba$ is uniformly behaved in $\N$ and completely additive. Then it is square-free uniformly behaved in $\N$.
\end{theorem}

\begin{proof}
Let $0<\theta <1$. 
Using the identity $\mu^{2} (n) =\sum_{m^{2} | n} \mu (m)$, we have that 
\begin{align*}
\frac{1}{N} \sum_{ 1\leq n\leq N, n\in \SF} e^{2\pi i a_{n} \theta}
&= \frac{1}{N} \sum_{n=1}^{N} \mu^{2}(n) e^{2\pi i a_{n} \theta}\\
&=\frac{1}{N} \sum_{n=1}^{N}\Big(\sum_{m^{2} | n} \mu (m)\Big) e^{2\pi i a_{n} \theta} \\
&=\frac{1}{N} \sum_{n=1}^{N}\sum_{1\leq m\leq \sqrt{n}} 1_{m^{2} | n} (m) \mu (m) e^{2\pi i a_{n} \theta} \\
& =\frac{1}{N} \sum_{m=1}^{[\sqrt{N}]}\mu (m) \sum_{n=1}^{[\frac{N}{m^{2}}]} e^{2\pi i a_{m^{2}n}\theta} +o(1)\\
&=\frac{1}{N} \sum_{m=1}^{[\sqrt{N}]}\mu (m) \sum_{n=1}^{[\frac{N}{m^{2}}]} e^{2\pi i (a_{n} +a_{m^{2}})\theta} +o(1)\\
&=\sum_{m=1}^{[\sqrt{N}]}\frac{\mu (m)}{m^{2}} \Big( \frac{1}{[\frac{N}{m^{2}}]}   \sum_{n=1}^{[\frac{N}{m^{2}}]} e^{2\pi i a_{n}\theta}\Big) e^{2\pi i a_{m^{2}}\theta}   +o(1)
\end{align*}

For any $\epsilon>0$, since $C=\sum_{m=1}^{\infty} 1/m^{2}$ converges, there is some $m_0$ such that for all $N>m_0^2$, 
\begin{align*}
& \hspace{15pt} \Big| \sum_{m=m_0+1}^{[\sqrt{N}]}\frac{\mu (m)}{m^{2}} \Big( \frac{1}{[\frac{N}{m^{2}}]}   \sum_{n=1}^{[\frac{N}{m^{2}}]} e^{2\pi i a_{n}\theta}\Big) e^{2\pi i a_{m^{2}}\theta}   +o(1)\Big|\\
&\leq \sum_{m=m_0+1}^{[\sqrt{N}]}\frac{|\mu (m)|}{m^{2}} \left| \frac{1}{[\frac{N}{m^{2}}]} \sum_{n=1}^{[\frac{N}{m^{2}}]} e^{2\pi i a_{n}\theta}\right| |e^{2\pi i a_{m^{2}}\theta}|   +o(1)\\
&\leq \sum_{m=m_0+1}^{[\sqrt{N}]}\frac{1}{m^2}+o(1)\\
&<\frac{\epsilon}{2}.
\end{align*}

Since 
$$
\lim_{N\to \infty} \frac{1}{N}   \sum_{n=1}^{N} e^{2\pi i a_{n}\theta} =0,
$$
we have an integer $N_0>m_{0}^{2}$ such that for all $N>N_0$ and for all $1\leq m\leq m_{0}$, 
$$
\left| \frac{1}{[\frac{N}{m^{2}}]}   \sum_{n=1}^{[\frac{N}{m^{2}}]} e^{2\pi i a_{n}\theta}\right| <\frac{\epsilon}{2C}.
$$
Thus for all $N>N_0$, 
\begin{align*}
&\hspace{15pt} \Big| \sum_{m=1}^{m_{0}}\frac{\mu (m)}{m^{2}} \Big( \frac{1}{[\frac{N}{m^{2}}]}   \sum_{n=1}^{[\frac{N}{m^{2}}]} e^{2\pi i a_{n}\theta}\Big) e^{2\pi i a_{m^{2}}\theta} \Big| \\
&\leq \sum_{m=1}^{m_{0}}\frac{|\mu (m)|}{m^{2}} \Big| \frac{1}{[\frac{N}{m^{2}}]}   \sum_{n=1}^{[\frac{N}{m^{2}}]} e^{2\pi i a_{n}\theta}\Big| |e^{2\pi i a_{m^{2}}\theta}|\\
&< \frac{\epsilon}{2A} \sum_{m=1}^{m_{0}} \frac{1}{m^{2}}\\ 
& <\frac{\epsilon}{2}.
\end{align*}
Thus, we have that for all $N>N_{0}$, 
\begin{align*}
\Big| \frac{1}{N} \sum_{1\leq n\leq N, n\in \SF} e^{2\pi i a_{n} \theta}\Big|  
& =\Big|\sum_{m=1}^{[\sqrt{N}]}\frac{\mu (m)}{m^{2}} \Big( \frac{1}{[\frac{N}{m^{2}}]}   \sum_{n=1}^{[\frac{N}{m^{2}}]} e^{2\pi i a_{n}\theta}\Big) e^{2\pi i a_{m^{2}}\theta}   +o(1)\Big|\\
&<\epsilon.
\end{align*}
Therefore,
$$
\lim_{N\to \infty}  \frac{1}{N} \sum_{1\leq n\leq N, n\in \SF} e^{2\pi i a_{n} \theta}=0.
$$ 
We proved the theorem.
 
\end{proof}

From the proofs of Theorem~\ref{sfubt} and Theorem~\ref{main}, we have that 

\medskip
\begin{theorem}~\label{fcor2}
 Suppose $\ba$ is uniformly behaved in $\N$ and completely additive. Suppose $f$ is $\ba$-MMLS, $\ba$-MMA, and minimal unqiuely ergodic. Then for any $\phi\in C(X)$, 
$$
\lim_{N\to \infty} \frac{1}{N} \sum_{1\leq n\leq N, n\in \SF} \phi (f^{a_{n}} x) =\frac{6}{\pi^{2}} \int_{K} \phi \; d\nu_{K}, \quad \forall x\in X,
$$
where $K$ is the $f$-minimal set such that $x\in \hbox{Basin}_{\ba}(K)$ and $\nu_{K}$ is the unique $(f|K)$-invariant probability measure.
 \end{theorem}

List $\SF$ as a subsequence ${\bf sf}=(s_{n})_{n\in \N}$ of $\N$. 

\medskip
\begin{theorem}~\label{sf}
The sequence ${\bf sf}$ is uniformly distributed modulo $1$ but not uniformly distributed in $\Z$. Thus, ${\bf sf}$ is not uniformly behaved in $\N$.
\end{theorem}

\begin{proof}
For any irrational number $0<\theta<1$, 
\begin{align*}
\lim_{N\to\infty} \frac{1}{N} \sum_{n=1}^{N} e^{2\pi i s_{n}\theta}
&=
\lim_{N\to\infty} \frac{s_{N}}{N} \frac{1}{s_{N}} \sum_{n=1}^{s_{N}} \mu^2(n)e^{2\pi i n\theta}\\
&=\frac{\pi^{2}}{6} \lim_{N\to \infty} \frac{1}{s_{N}} \sum_{n=1}^{s_{N}}\Big(\sum_{m^{2} | n} \mu (m)\Big) e^{2\pi i n \theta}.
\end{align*}
Consider 
\begin{align*}
\frac{1}{s_{N}} \sum_{n=1}^{s_{N}}\Big(\sum_{m^{2} | n} \mu (m)\Big) e^{2\pi i n \theta}
&=\frac{1}{s_{N}} \sum_{n=1}^{s_{N}}\sum_{1\leq m\leq \sqrt{n}} 1_{m^{2} | n} (m) \mu (m) e^{2\pi i n \theta} \\
& =\frac{1}{s_{N}} \sum_{m=1}^{[\sqrt{s_{N}}]}\mu (m) \sum_{n=1}^{[\frac{s_{N}}{m^{2}}]} e^{2\pi i (m^{2}n)\theta} +o(1)\\
& =\sum_{m=1}^{[\sqrt{s_{N}}]}\frac{\mu(m)}{m^2}\left(\frac{1}{[\frac{s_{N}}{m^2}]}\sum_{n=1}^{[\frac{s_{N}}{m^2}]}e^{2\pi i n(m^2\theta)}\right) +o(1).
\end{align*}
For every $\epsilon >0$, since $C=\sum_{m=1}^{\infty} 1/m^2<\infty$, we have a $m_{0}>0$ such that for any $s_{N}>m_{0}^{2}$, 
\begin{align*}
& \hspace{15pt} \Big| \sum_{m=m_{0}+1}^{[\sqrt{s_{N}}]}\frac{\mu(m)}{m^2}\left(\frac{1}{[\frac{s_{N}}{m^2}]}\sum_{n=1}^{[\frac{s_{N}}{m^2}]}e^{2\pi i n(m^2\theta)}\right) +o(1)\Big|\\
&\leq  \sum_{m=m_{0}+1}^{[\sqrt{s_{N}}]}\frac{|\mu(m)|}{m^2}\left|\frac{1}{[\frac{s_{N}}{m^2}]}\sum_{n=1}^{[\frac{s_{N}}{m^2}]}e^{2\pi i n(m^2\theta)}\right| +o(1) \\
&\leq  \sum_{m=m_{0}+1}^{[\sqrt{s_{N}}]}\frac{1}{m^2}+o(1)\\ &<\frac{\epsilon}{2}.
\end{align*}
Since  
\begin{equation}~\label{irr}
\lim_{N\to\infty} \frac{1}{[\frac{s_{N}}{m^2}]}\sum_{n=1}^{[\frac{s_{N}}{m^2}]}e^{2\pi i n(m^2\theta)} =0
\end{equation}
for all $1\leq m\leq m_{0}$, we have an integer $N_{0}>0$ such that for all $N>N_{0}$ and $1\leq m\leq m_{0}$,
$$
\Big| \frac{1}{[\frac{s_{N}}{m^2}]}\sum_{n=1}^{[\frac{s_{N}}{m^2}]}e^{2\pi i n(m^2\theta)}\Big| <\frac{\epsilon}{2C}.
$$
Then we have that 
\begin{align*}
\Big| \sum_{m=1}^{m_{0}} \frac{\mu(m)}{m^2}\Big(\frac{1}{[\frac{s_{N}}{m^2}]}\sum_{n=1}^{[\frac{s_{N}}{m^2}]}e^{2\pi i n(m^2\theta)}\Big)\Big|
&\leq \sum_{m=1}^{m_{0}} \frac{|\mu(m)|}{m^2}\Big| \frac{1}{[\frac{s_{N}}{m^2}]}\sum_{n=1}^{[\frac{s_{N}}{m^2}]}e^{2\pi i n(m^2\theta)}\Big|\\
&\leq \frac{\epsilon}{2C} \sum_{m=1}^{m_{0}} \frac{1}{m^2}\\
&<\frac{\epsilon}{2}.
\end{align*}
This implies that for any $N>N_{0}$,
\begin{align*}
\Big| \frac{1}{s_{N}} \sum_{n=1}^{s_{N}} \mu^2(n)e^{2\pi i n\theta}\Big| 
&=\Big| \sum_{m=1}^{[\sqrt{s_{N}}]}\frac{\mu(m)}{m^2}\left(\frac{1}{[\frac{s_{N}}{m^2}]}\sum_{n=1}^{[\frac{s_{N}}{m^2}]}e^{2\pi i n(m^2\theta)}\right) +o(1)\Big| \\
&<\epsilon.
\end{align*}
Therefore, we get
$$
\lim_{N\to\infty} \frac{1}{N} \sum_{n=1}^{N} e^{2\pi i s_{n}\theta}=0.
$$
This says that ${\bf sf}$ is uniformly distributed modulo $1$. However, the last limit does not hold for all rational numbers $0<\theta <1$; 
for example, for $\theta=1/m^{2}$, the limit in (\ref{irr}) is $1$. One can see (for example, refer to~\cite{Nu}) that 
$$
\lim_{N\to \infty} \frac{\# (\{1\leq n\leq N \;|\; s_{n}=r\pmod{m}\}}{N} =\frac{1}{m}\prod_{p|m} (1-\frac{1}{p^{2}})^{-1}
$$
and $\prod_{p|m} (1-1/p^{2})^{-1}\not=1$ in general. This implies that ${\bf sf}$ is not uniformly distributed in $\Z$ and, thus, is not uniformly behaved in $\N$. 
\end{proof}

Recall that we use $EF$ to denote the set of all natural numbers whose prime factors are even and $OF$ to denote the set of all natural numbers whose prime factors are odd.  
List $EF\cap \SF$ as a subsequence of $\N$
$$
{\bf efs}=(efs_{n})_{n\in \N}
$$
and list $OF\cap \SF$ as a subsequence of $\N$ 
$$
{\bf ofs} =(ofs_{n})_{n\in \N}.
$$

\medskip
\begin{corollary}~\label{sfa}
 The sequences ${\bf efs}$ and ${\bf ofs}$ are both uniformly distributed modulo $1$ but not uniformly distributed in $\Z$. Thus, neither of them behaves uniformly in $\N$. 
\end{corollary} 

\begin{proof}
Let us only prove for ${\bf efs}$. The proof for ${\bf ofs}$ is similar. Denote $a_{n}=efs_{n}$.

The indicator of $EF\cap \SF$ is $t_{n}=0$ for $n\not\in \SF$ and 
$$
t_{n} =\frac{1+\mu (n)}{2} \quad \hbox{if $n\in \SF$}.  
$$

For any $0<\theta<1$, 
\begin{align*}
\lim_{N\to \infty} \frac{1}{N} \sum_{n=1}^{N} e^{2\pi i a_{n}\theta} 
&= \lim_{N\to \infty} \frac{1}{N} \sum_{ 1\leq n \leq a_{N}} t_{n} e^{2\pi i n \theta} \\
&= \lim_{N\to \infty} \frac{1}{N} \sum_{1\leq n \leq a_{N}, n\in \SF} \frac{1+\mu (n)}{2} e^{2\pi i n \theta} \\
&= \lim_{N\to \infty} \frac{a_{N}}{2N} \frac{1}{a_{N}} \sum_{1\leq n\leq a_{N}, n\in \SF} \mu (n) e^{2\pi i n \theta} \\
&+ \lim_{N\to \infty}  \frac{a_{N}}{2N} \frac{1}{a_{N}} \sum_{1\leq n\leq a_{N}, n\in \SF} e^{2\pi i n \theta} \\
&= \Big(\lim_{N\to \infty} \frac{a_{N}}{2N}\Big) \Big(\lim_{N\to \infty} \frac{1}{a_{N}} \sum_{n=1}^{a_{N}} \mu (n) e^{2\pi i n \theta} \Big)\\
&+ \Big(\lim_{N\to \infty}  \frac{a_{N}}{2N}\Big) \Big(\lim_{N\to \infty} \frac{1}{a_{N}} \sum_{n=1}^{a_{N}} \mu^{2}(n) e^{2\pi i n \theta} \Big).
\end{align*}
Since 
$$
\lim_{N\to \infty}  \frac{a_{N}}{2N} = \frac{1}{2D(EF\cap \SF)}\not=0
$$
and since $\mu(n)$ is an oscillating sequence,
$$
\Big(\lim_{N\to \infty} \frac{1}{a_{N}} \sum_{n=1}^{a_{N}} \mu(n) e^{2\pi i n \theta} \Big)=0, 
$$
We get that 
$$
\Big(\lim_{N\to \infty} \frac{a_{N}}{2N}\Big) \Big(\lim_{N\to \infty} \frac{1}{a_{N}} \sum_{n=1}^{a_{N}} \mu (n) e^{2\pi i n \theta} \Big)=0.
$$ 
As we have seen in the proof of Theorem~\ref{sf}, when $0<\theta<1$ is irrational,  
$$
\lim_{N\to \infty} \frac{1}{a_{N}} \sum_{n=1}^{a_{N}} \mu^{2}(n) e^{2\pi i n \theta}=0.
$$
Thus,
$$
\lim_{N\to \infty} \frac{1}{N} \sum_{n=1}^{N} e^{2\pi i a_{n}\theta}=0.
$$
This says that $\ba$ is uniformly distributed modulo $1$. 
There are rational numbers $0<\theta<1$ such that 
$$
\lim_{N\to \infty} \frac{1}{a_{N}} \sum_{n=1}^{a_{N}} \mu^{2}(n) e^{2\pi i n \theta}\not=0
$$
as we have seen in the proof of Theorem~\ref{sf}. 
This says that ${\bf efs}$ is not uniformly distributed in $\Z$.
Thus, we proved that ${\bf efs}$ is not uniformly behaved in $\N$.
\end{proof}

As suggested by the proof of Theorem~\ref{sf}, where we used the completely multiplicative property, we have the following proposition, whose proof is similar to that of Theorem~\ref{sf} just by replacing $n$ by $a_{n}$ and by using $a_{nm}=a_{n}a_{m}$.

\medskip
\begin{definition}~\label{cm}
 We say a sequence $\ba$ taking values in $\N$ is completely multiplicative if $a_{mn} =a_{n}a_{m}$ for all natural numbers $n$ and $m$.  
\end{definition}
  
The sequence ${\bf n} =(n)_{n\in \N}$ is completely multiplicative.  
 
\medskip
\begin{corollary}~\label{sfubc}
Suppose $\ba$ is uniformly behaved in $\N$ and completely multiplicative. Then it is square-free uniformly distributed modulo $1$. 
\end{corollary}

The sequence ${\bf n}$ is uniformly behaved in $\N$ and completely multiplicative. Theorem~\ref{sf} says ${\bf n}$ is square-free uniformly distributed modulo $1$, but not square-free uniformly distributed in $\Z$. 
However, Corollary~\ref{sfubc} says that if a uniformly behaved in $\N$ and completely multiplicative sequence and if it is also the square-free uniformly distributed in $\Z$, then it is square-free uniformly behaved in $\N$. Thus, we have the following result.

\medskip
\begin{theorem}~\label{sfc}
 Suppose $\ba$ is uniformly behaved in $\N$ and completely multiplicative, and suppose it is also square-free uniformly distributed in $\Z$. Suppose $f$ is $\ba$-MMLS, $\ba$-MMA, and minimally uniquely ergodic. Then for any $\phi\in C(X)$, 
$$
   \lim_{N\to \infty} \frac{1}{N} \sum_{1\leq n\leq N, n\in \SF} \phi (f^{a_{n}} x) =\frac{6}{\pi^{2}} \int_{K} \phi d\nu_{K}, \quad \forall \; x\in X,
$$
where $K\subset X$ is the $f$-minimal set such that $x\in Basin_{\ba}(K)$ and $\nu_{K}$ is the unique $(f|K)$-invariant probability measure.
 \end{theorem}

\begin{proof}
We can not follow the proof of Theorem~\ref{sf} as we did for Corollary~\ref{sfubc}. We only need to give a new proof for Step 3 as in the proof of Theorem~\ref{main}. Steps 1 and 2, as in the proof of Theorem~\ref{main}, apply here. 

Let us prove Step 3, that is, for any $f$-minimal set $K$,
$$
   \lim_{N\to \infty} \frac{1}{N} \sum_{1\leq n \leq N, n\in \SF} \phi (f^{a_{n}} x)  =\frac{6}{\pi^{2}} \int_{K} \phi d\nu_{K}, \quad \forall\; x\in K. 
$$
 
Consider $\SF=(s_{n})_{n\in \N}$. Let $b_{n}=a_{s_{n}}$. From Corollary~\ref{sfubc} and the assumption that $\ba$ is square-free uniformly distributed in $\Z$, for any $0<\theta<1$, we have 
\begin{align*}
\lim_{N\to \infty} \frac{1}{N} \sum_{n=1}^{N} e^{2\pi i b_{n} \theta}
&=\Big( \lim_{N\to \infty} \frac{s_{N}}{N}\Big)\Big( \lim_{N\to \infty} \frac{1}{s_{N}} \sum_{1\leq n\leq s_{N}, n\in \SF} e^{2\pi i a_{n} \theta}\Big)\\
&= \frac{\pi^{2}}{6} \times 0=0.
\end{align*}
This implies that ${\bf b} =(b_{n})_{n\in \N}$ is uniformly behaved in $\N$. Theroem~\ref{main} implies that, for any $\phi \in C(K)$, 
$$
\lim_{N\to \infty} \frac{1}{N} \sum_{n=1}^{N} \phi (f^{b_{n}} x)  =\int_{K} \phi \; d\nu_{K}, \quad \forall\; x\in K.
$$
Let $M=\#(\SF\cap [1,N])$
\begin{align*}
\lim_{N\to \infty} \frac{1}{N} \sum_{1\leq n \leq N,n\in \SF}
\phi (f^{a_{n}} x) 
&= \lim_{N\to \infty} \frac{M}{N} \frac{1}{M}  \sum_{n=1}^{M} \phi (f^{a_{s_{n}}} x) \\
&=\Big( \lim_{N\to \infty} \frac{M}{N} \Big) \Big(\lim_{M\to \infty}  \frac{1}{M}  \sum_{n=1}^{M} \phi (f^{b_{n}} x)\Big)\\
&=\frac{6}{\pi^{2}} \int_{K} \phi d\; \nu_{K}.
\end{align*}
It completes the proof. 
\end{proof}

\medskip
\medskip

 \end{document}